\documentclass{amsart}
\usepackage{graphicx}

\newtheorem{thm}{Theorem}[section]
\newtheorem{cor}[thm]{Corollary}

\theoremstyle{definition}
\newtheorem{defn}[thm]{Definition}
\newtheorem{asum}[thm]{Assumption}
\theoremstyle{remark}
\newtheorem{rem}[thm]{Remark}
\newtheorem{exam}[thm]{Example}

\numberwithin{equation}{section}

\begin{document}

\title[Level-crossings of random walks]{Level-crossings of symmetric random walks and their application}%
\author{Vyacheslav M. Abramov}%
\address{Department of Mathematics and Statistics, The University of Melbourne,
Parkville,
Victoria 3010, Australia}%
\email{vabramov126@gmail.com}%


\begin{abstract}
Let $X_1$, $X_2$, $\ldots$ be a sequence of independently and
identically distributed random variables with $\mathsf{E}X_1=0$, and
let $S_0=0$ and $S_t=S_{t-1}+X_t$, $t=1,2,\ldots$, be a random walk.
Denote $\tau=\begin{cases}\inf\{t>1: S_t\leq0\}, &\text{if} \ X_1>0,\\
1, &\text{otherwise}.
\end{cases}$ Let $\alpha$ denote a positive number,
and let $L_\alpha$ denote the number of level-crossings from the
below (or above) across the level $\alpha$ during the interval $[0,
\tau]$. Under quite general assumption, an inequality for the
expected number of level-crossings is established. Under some
special assumptions, it is proved that there exists an infinitely
increasing sequence $\alpha_n$ such that the equality
$\mathsf{E}L_{\alpha_n}=c\mathsf{P}\{X_1>0\}$ is satisfied, where
$c$ is a specified constant that does not depend on $n$. The result
is illustrated for a number of special random walks. We also give
non-trivial examples from queuing theory where the results of this
theory are applied.
\end{abstract}
\maketitle
\section{Introduction}

\noindent In this article, we discuss beautiful properties of
symmetric random walks. All random walks considered in this article
are assumed to be one-dimensional.

A symmetric random walk is a well-known object in probability
theory, and there are many classic books such as \cite{Feller},
\cite{Spitzer} that give a very detailed its study. Nevertheless,
even at elementary level, symmetric random walks are very appealing
and have astonishing properties of their level-crossings.

Let $X_1$, $X_2$, $\ldots$ be a sequence of independent random
variables taking the values $+1$ and $-1$ each with probability
$\frac{1}{2}$. The \textit{simplest symmetric random walk} is
defined as $S_0=0$, and $S_n=X_1+X_2+\ldots+X_n$. Let
$\tau=\inf\{i>0: S_i=0\}$ be a stopping time. For integer positive
$\alpha$, let $L_\alpha$ denote the total number of events
$\{X_t=\alpha-1$ and $X_{t+1}=\alpha\}$ that occurs during the
interval $[0,\tau]$, i.e. for $t=0,1,\ldots,\tau-1$. The remarkable
property of this random walk is that
$\mathsf{E}L_\alpha=\frac{1}{2}$ for all $\alpha$.

In a book \cite{Szekeley}, this property is classified as a paradox
in probability theory. The proof of this level-crossing property (in
slightly different formulation) can be found in \cite{Wolff-book},
p. 411, where that proof is a part of a special theory and seems to
be complicated. The level-crossing properties of symmetric random
walks are very important in many applications of Applied
Probability. For instance, the aforementioned level-crossing
property for symmetric random walk is directly reformulated in terms
of the $M/M/1/n$ queuing system. If the expectation of interarrival
and service times in that queuing system are equal, then the
expected number of losses during a busy period is equal to 1 for all
$n\geq0$. Surprisingly, this property holds true for the more
general $M/GI/1/n$ queuing system as well (see \cite{Abramov},
\cite{Righter}, \cite{Wolff} and a survey paper \cite{Abramov3} for
further information).

In the present paper, we study properties of level-crossings for
general symmetric random walks, which are defined as $S_0=0$ and
$S_n=X_1+X_2+\ldots+X_n$, where $X_1$, $X_2$, $\ldots$ are
independently and identically distributed random variables with
$\mathsf{E}X_1=0$. The exact formulation of the problem and relevant
definitions are given later.

There is a huge number of papers, where the level-crossings are used
and serve as a main tool of analysis. We refer only a few papers
that have a theoretical contribution in the areas. Special questions
on asymptotic behavior of crossings moving boundaries have been
studied in \cite{Novikov} (see also \cite{Shiryayev}, p. 536 as well
as the references in \cite{Novikov} about previous studies). The
asymptotic number of crossings that are required to reach a high
boundary has been studied in \cite{Siegmund}. Asymptotic behavior of
random walks with application to statistical theory has been studied
in \cite{Burridge}. Level-crossings for Gaussian random fields have
been studied in \cite{Hasofer} and \cite{Adler and Hasofer} and
those for Markov and stationary processes in \cite{Borovkov and Last
2008} and \cite{Borovkov and Last 2010}.

The present paper addresses an open question related to general
symmetric random walks: \textit{whether or not the aforemention
property of symmetric random walk is valid for general random walk}?
Having a negative answer on this question under the general setting,
in the paper we find the conditions under which the aforementioned
result on the simplest symmetric random walk can be extended to more
general symmetric random walks, i.e. conditions when the expected
number of level-crossings remains unchanged when the level $\alpha$
varies.

The article is organized as follows. In short Section \ref{Sect2} we
classify symmetric random walks, which is useful for further
presentation of the results. In Section \ref{Sect3} we give
necessary definitions, examples and specifically a counterexample
showing that the aforementioned property of level-crossings being
correct for the simplest symmetric random walk is no longer valid
for general symmetric random walks. In Section \ref{Sect4} we prove
the main results of this paper on level-crossings in general
symmetric random walks. In Section \ref{Sect5} we discuss a
nontrivial application of these results in queuing theory. In
Section \ref{Sect6} we conclude the paper.

\section{Classification of symmetric random walks}\label{Sect2}

\noindent Let $X_1$, $X_2$,\ldots be a sequence of independently and
identically distributed random variables, and let $S_0=0$ and
$S_{t+1}=S_t+X_{t+1}$, $t=0,1\ldots$ The sequence $\{S_t\}$ is
called \textit{random walk}. A random walk is called
\textit{$E$-symmetric} if $\mathsf{E}X_1=0$. If, in addition,
$\mathsf{P}\{X_1>0\}=\mathsf{P}\{X_1<0\}$, then the random walk is
called \textit{$P$-symmetric}. A random walk is called
\textit{purely symmetric} if for any $x$, $\mathsf{P}\{X_1\leq
x\}=\mathsf{P}\{X_1\geq -x\},$ and $\mathsf{E}|X_1|<\infty$.

Apparently, any purely symmetric random walk is $P$-symmetric, and
any $P$-symmetric random walk is $E$-symmetric, i.e.
$$
\text{Purely symmetric RW}\Longrightarrow P^\_\text{Symmetric
RW}\Longrightarrow E^\_\text{Symmetric RW},
$$

In the case where $X_1$ takes the values $+1$ or $-1$ with the equal
probability $\frac{1}{2}$, the random walk is called
\textit{simplest symmetric random walk}.

\section{Definitions, examples and counterexamples}\label{Sect3}

\noindent
\begin{defn}\label{defn1} The stopping time for the random walks is as follows:

$$\tau=\begin{cases}\inf\{t>1: S_t\leq0\}, &\text{if} \ X_1> 0,\\
1, &\text{if} \ X_1\leq0.\\
\end{cases}$$

\end{defn}

\begin{defn}\label{defn2} For any positive $\alpha$, by the number
of level-crossings across the level $\alpha$ we mean the total
number of events $\{S_{t-1}<\alpha$ and $S_t\geq\alpha\}$, where the
index $t$ runs the integer values from 1 to $\tau$. The number of
level-crossings across the level $\alpha$ is denoted $L_\alpha$.
\end{defn}

We start from the elementary example for the following purely
symmetric random walk.
\begin{exam}\label{exam1}
Let $X_1$ take values $\{-1, -2, +1, +2\}$ each with the probability
$\frac{1}{4}$. Taking the level $\alpha=1$ it is easy to see that
the expected number of level-crossings is equal to $\frac{1}{2}$
exactly. Indeed, there is probability $\frac{1}{2}$ that $X_1$ is
negative and $\frac{1}{2}$ that it is positive. So, if $X_1>0$, then
the value $S_1$ is not smaller than $1$. The level $1$ is once
reached immediately, and the counter of level-crossings is set to 1.
After this, the excursion of the random walk will be always above
the point $1$ until the time $\tau-1$. During the time interval
$[1,\tau-1]$ this point can be reached from the above only, but not
from the below. Finally, until the stopping time $\tau$, the level
$\alpha=1$ is no longer reached or intersected from the below. In
this case, the total expectation formula gives
$\mathsf{E}L_1=\frac{1}{2}$. As we see, this case is in agreement
with the result in the case of the simplest symmetric random walk,
where $\mathsf{E}L_n=\frac{1}{2}$ for all $n\geq1$.
\end{exam}

Thus, in Example \ref{exam1} we obtain $\mathsf{E}L_1=\frac{1}{2}$.
Is it true that $\mathsf{E}L_\alpha=\frac{1}{2}$ for all positive
$\alpha$ as well? Unfortunately, by direct calculations it is hard
to check this property even for $\alpha=2$. We leave this question
now, but answer it later.

\begin{exam}\label{exam3}
Consider another example, related now to a $P$-symmetric random
walk. Assume that $X_1$ takes the value $2$ with probability
$\frac{1}{2}$, the value $-1$ with probability $\frac{1}{4}$ and the
value $-3$ with probability $\frac{1}{4}$. Apparently, in the case
 $\alpha=1$  the expected number of level-crossings from
the below across this level is equal to $\frac{1}{2}$. As in the
example above, the level 1 is intersected in the first step of the
random walk (if $X_1$ is positive), and the following excursion is
always above this level before the time $\tau-1$. Using the total
expectation formula, as in the case of Example \ref{exam1}, we
obtain $\mathsf{E}L_1=\frac{1}{2}$. For the same random walk, assume
now that $\alpha=2$. If $X_1$ is positive, then the level 2 is
reached immediately. However, there is the positive probability that
the random walk will return to the level 1 and then intersect the
level 2 once again. Hence, $\mathsf{E}L_2>\frac{1}{2}$.
\end{exam}

Thus, $\mathsf{E}L_\alpha$ depends on $\alpha$ in general. Following
this, in the present paper there are considered two main questions
associated with the behaviour of $\mathsf{E}L_\alpha$ when $\alpha$
varies. First, \textit{under what conditions $\mathsf{E}L_\alpha$
remains the same when $\alpha$ varies?} Second, \textit{is
$\mathsf{P}\{X_1>0\}$ the minimum of all possible values of
$\mathsf{E}L_\alpha$? For what family of symmetric random walks the
last is true?}

\section{Level-crossings of $E$-, $P$- and purely symmetric random walks}\label{Sect4}
\noindent In this section we establish the properties of the
level-crossings for $P$-, $E$- and purely symmetric random walks,
and thus answer on the questions formulated in Section \ref{Sect3}.

According to well-known results in probability theory (such as the
second Borel-Cantelli lemma and Markov property, for instance), it
is easy to conclude that any $E$-symmetric random walk is recurrent
in the sense that $\mathsf{P}\{\tau<\infty\}=1$.

In the following we use the following notation. A random variable
$X_t$, $t=1,2,\ldots$, is represented
$$
X_t=\begin{cases}X_t^+, &\text{if} \ X_t>0,\\
X_t^-, &\text{if} \ X_t\leq0,
\end{cases}
$$
where $X_t^+$ takes the only positive values of $X_t$, while $X_t^-$
takes the nonpositive values of $X_t$.

 The sequences
$\{X_t^+\}$ and $\{X_t^-\}$, $t=1,2,\ldots$ are independent and
consist of independently and identically distributed random
variables with the expectations $\frac{a}{\mathsf{P}\{X_t>0\}}$ and
$\frac{a}{\mathsf{P}\{X_t\leq0\}}$, respectively.
 As well, we denote
$S_0^+=S_0^-=0$ and correspondingly
$S_t^+=S_{t-1}^++X_t\mathsf{I}\{X_t>0\}$ and
$S_t^-=S_{t-1}^--X_t\mathsf{I}\{X_t\leq0\}$, so $S_t=S_t^+-S_t^-$.

Denote by $t_1(\alpha)$ the first time during the time interval
$[0,\tau]$ (if any) such that $S_{t_1(\alpha)}\geq\alpha$, and by
$\tau_1(\alpha)$ denote the first time after $t_1(\alpha)$ such that
$S_{\tau_1(\alpha)}<\alpha$. Note, that existence of the time
$\tau_1(\alpha)$ is associated with the existence of the time
$t_1(\alpha)$. If level $\alpha$ is not reached in the interval
$[0,\tau]$, then the time $\tau_1(\alpha)$ does not exist either.
Next, let $t_2(\alpha)$ be the first time after $\tau_1(\alpha)$ and
during the time interval $[0,\tau]$ such that
$S_{t_2(\alpha)}\geq\alpha$, and let $\tau_2(\alpha)$ be the first
time after $t_2(\alpha)$ such that $S_{\tau_2(\alpha)}<\alpha$. As
above, the existence of $t_2(\alpha)$ is associated with that of
$t_1(\alpha)$, and, in turn, the existence of $\tau_2(\alpha)$ is
associated with that of $t_2(\alpha)$. The times $t_i(\alpha)$ and
$\tau_i(\alpha)$ ($i>2$) are defined similarly.

Note, that the existence of $t_2(\alpha)$ and, respectively,
$\tau_2(\alpha)$ generally depend on $\alpha$ as well. In Examples
\ref{exam1} and \ref{exam3} for the specific level $\alpha=1$ the
second level-crossing does not exist. So, in the following we reckon
that existence of $t_i(\alpha)$ and $\tau_i(\alpha)$ for $i=2$ (and
hence for $i\geq2$) are guaranteed with choice of level $\alpha$,
that is, the process is assumed to be defined in the probability
space $\{\Omega, \mathcal{F}, {\bf
F}=(\mathcal{F}_\alpha)_{\alpha\geq\alpha_0}, \mathsf{P}\}$ with an
increasing family of filtrations $\mathcal{F}_\alpha$ such that
$\frak{A}_2\in\mathcal{F}_{\alpha_0}$, and
$\mathsf{P}\{\frak{A}_2\}>0$.

\begin{asum}\label{asum1} Let $\frak{A}_i(\alpha)$ denote the event
``the $i$th crossing of the level $\alpha$ occurs during the time
interval $[0,\tau]$", and assume that
$\mathsf{E}\{S_{t_{1}(\alpha)-1}|\frak{A}_1(\alpha)\}>0$,
  $\mathsf{E}\{S_{\tau_1(\alpha)}|\frak{A}_1(\alpha)\}>0$,

\begin{equation}\label{OS1}
\mathsf{E}\{S_{t_{2}(\alpha)-1}|\frak{A}_2(\alpha)\}=\mathsf{E}\{S_{t_{1}(\alpha)-1}|\frak{A}_1(\alpha)\},
\end{equation}
and
\begin{equation}\label{OS2}
\mathsf{E}\{S_{\tau_2(\alpha)}|\frak{A}_2(\alpha)\}=\mathsf{E}\{S_{\tau_1(\alpha)}|\frak{A}_1(\alpha)\}.
\end{equation}
\end{asum}
\smallskip

\begin{rem} In general, it is hard to check Conditions \eqref{OS1} and
\eqref{OS2}. However, in some cases conditions \eqref{OS1} and
\eqref{OS2} are satisfied automatically. This is demonstrated in the
two examples given below.
\end{rem}

\begin{exam}\label{exam5}
Consider the following $E$-symmetric random walk. Assume that $X_t$
takes values 1 with probability $\frac{2}{3}$ and $-2$ with
probability $\frac{1}{3}$. Take $\alpha\geq2$, and assume for
convenience that $\alpha$ is integer. Then, it is readily seen that
under an occurrence of the event $\frak{A}_i(\alpha)$
($i=1,2,\ldots$), we always have $S_{t_i(\alpha)-1}=\alpha-1$.
Moreover, $\mathsf{E}\{S_{\tau_1(\alpha)}|\frak{A}_1(\alpha)\}>0$,
and
$\mathsf{E}\{S_{\tau_1(\alpha)}|\frak{A}_1(\alpha)\}=\mathsf{E}\{S_{\tau_2(\alpha)}|\frak{A}_2(\alpha)\}$.
\end{exam}

\begin{exam}\label{exam6} Assume
that $X_t^+$ takes a single positive value $1$ and $(-X_t^-)$ is a
geometrically distributed random variable with mean $m$. Assume that
$\mathsf{P}\{X_t>0\}=\frac{m}{1+m}$. Then, $\mathsf{E}X_t=0$, and it
is the convention. Apparently, that for any given integer level
$\alpha\geq\alpha_0$ the event $\frak{A}_2(\alpha)$ occurs, than
 $S_{t_1(\alpha)-1}$ and
$S_{t_2(\alpha)-1}$ both are equal to $\alpha-1$. In addition,
$\mathsf{E}\{S_{\tau_1(\alpha)}|\frak{A}_1(\alpha)\}=\mathsf{E}\{S_{\tau_2(\alpha)}|\frak{A}_2(\alpha)\}$.
The value $\alpha_0$ is assumed to be chosen such that both
$\mathsf{E}\{S_{t_1(\alpha)-1}|\frak{A}_1(\alpha)\}>0$ and
$\mathsf{E}\{S_{\tau_1(\alpha)}|\frak{A}_1(\alpha)\}>0$.  That is we
set $\alpha_0=\max\left[1, -\mathsf{E}\{S_{\tau}|X_1>0\}\right]$.
\end{exam}

Examples \ref{exam5} and \ref{exam6} fall into the category of a
special class of random walks where $X_t^+$ takes only one positive
value $d$, while $X_t^-$ takes the values $\{0, -d, -2d, \ldots\}$,
some of them can have probability 0. This class is considered later
by Theorem \ref{thm4}.

The next example demonstrates the case where Assumption \ref{asum1}
is not satisfied.

\begin{exam}\label{exam7}
Consider a symmetric random walk where $X_t^+$ takes a single
positive value $1$, while $(-X_t^-)$ is an exponentially distributed
random variable with mean 1. According to the well-known property of
exponential distribution,
$\mathsf{E}\{S_{\tau_1(\alpha)}|\frak{A}_1(\alpha)\}=\mathsf{E}\{S_{\tau_2(\alpha)}|\frak{A}_2(\alpha)\}
=\alpha-1$ ($\alpha$ is assumed to be greater than 1). However, it
is readily seen that generally
$\mathsf{E}\{S_{t_1(\alpha)-1}|\frak{A}_1(\alpha)\}\neq\mathsf{E}\{S_{t_2(\alpha)-1}|\frak{A}_2(\alpha)\}$.
Indeed, let $\alpha$ be integer, say 2. Then, clearly
$\mathsf{P}\{S_{t_1(2)-1}=1|\frak{A}_1(2)\}>0$, while
$\mathsf{P}\{S_{t_2(2)-1}=1|\frak{A}_2(2)\}=0$. Hence, it is easy to
find that $\mathsf{E}\{S_{t_2(2)-1}|\frak{A}_2(2)\}=\frac{3}{2}$,
while $\mathsf{E}\{S_{t_1(2)-1}|\frak{A}_1(2)\}<\frac{3}{2}$.
\end{exam}

\begin{thm}\label{thm2}
Let $S_t$ be a $E$-symmetric random walk, and let Assumption
\ref{asum1} be satisfied for some $\alpha>0$. Then, there exists an
infinitely increasing sequence of levels $\alpha_n$ such that
\begin{equation}\label{(3)}
\mathsf{E}L_{\alpha_n}=\mathsf{E}L_{\alpha}=\frac{a-\mathsf{E}\{S_\tau|X_1>0\}}{a-b}\mathsf{P}\{X_1>0\},
\end{equation}
where $a=\mathsf{E}\{X_1|X_1>0\}$ and
$b=\mathsf{E}\{S_{t_1(\alpha)-1}|\frak{A}_1(\alpha)\}-\mathsf{E}\{S_{\tau_1(\alpha)}|\frak{A}_1(\alpha)\}$.
\end{thm}

\begin{proof}  According to the total
expectation formula
\begin{equation*}
\begin{aligned}
\mathsf{E}L_\alpha&=\mathsf{E}\{L_\alpha|X_1>0\}\mathsf{P}\{X_1>0\}
+\mathsf{E}\{L_\alpha|X_1\leq0\}\mathsf{P}\{X_1\leq0\}\\
&=\mathsf{E}\{L_\alpha|X_1>0\}\mathsf{P}\{X_1>0\}.
\end{aligned}
\end{equation*}
Hence, the challenge is to prove first that
$\mathsf{E}\{L_\alpha|X_1>0\}=\frac{a-\mathsf{E}\{S_\tau|X_1>0\}}{a-b}$
for a given $\alpha>0$, and then to build an increasing sequence of
$\alpha_n$: $\alpha<\alpha_1<\alpha_2<\ldots$ where
$\mathsf{E}L_{\alpha_n}=\mathsf{E}L_{\alpha}$.

Let $t_1$, $t_2$,\ldots, $t_{L_{\alpha}}$ be a set of times where
the events $\{S_{t_i-1}<\alpha$ and $S_{t_i}\geq \alpha\}$ occur,
and, respectively, let $\tau_1$, $\tau_2$,\ldots,
$\tau_{L_{\alpha}}$ be a set of times where
$\{S_{\tau_i-1}\geq\alpha$ and $S_{\tau_i}<\alpha\}$. The notation
for $t_1$, $t_2$,\ldots and $\tau_1$, $\tau_2$, \ldots  is similar
to the above, but the parameter $\alpha$ is omitted from the
notation for the sake of convenience.

Taking into consideration that the sequence $\{S_t\}$ is a
martingale and Assumption \ref{asum1} is satisfied, we have the
following properties:
\begin{equation}\label{1}
\mathsf{E}\{S_{t_{i}-1}|\frak{A}_i\}=\mathsf{E}\{S_{t_{j}-1}|\frak{A}_j\},
\end{equation}
and
\begin{equation}\label{2}
\mathsf{E}\{S_{\tau_i}|\frak{A}_i\}=\mathsf{E}\{S_{\tau_j}|\frak{A}_j\},
\end{equation}
where $i,j=1,2,\ldots, t_{L_{\alpha}}$.

Note, that the expectations that defined by \eqref{1} and \eqref{2}
need not be positive in general, because the random value $S_\tau$
is non-positive, and $\mathsf{E}\{S_\tau|X_1>0\}\leq0$. So,
Assumption \ref{asum1}, where these expectations are assumed to be
positive, implies the choice of $\alpha$ for which this assumption
is satisfied. As well, $S_{\tau-1}>0$, $X_\tau<0$ with probability
1. Then we obtain
$$
\mathsf{E}\{S_\tau|X_1>0\}=\mathsf{E}\{S_{\tau-1}|X_1>0\}+\mathsf{E}\{X_\tau|X_1>0\}.
$$
Hence, for $\alpha\geq -\mathsf{E}\{S_\tau|X_1>0\}$ the expectations
that defined by \eqref{1} and \eqref{2} are guaranteed to be
positive.

These properties enable us to establish easily level-crossing
properties of any $E$-symmetric random walk where Assumption
\ref{asum1} is satisfied. Let us scale the original time interval
$[0, \tau]$ by deleting the time intervals $[t_{i}-1,\tau_i)$,
$i=1,2,\ldots,L_\alpha$ and merging the corresponding ends. As it
done, $S_{t_{i}-1}$ and $S_{\tau_i}$ take distinct values in
general, and the difference between their expectations in these ends
is denoted to be equal to $b$, i.e.
$b=\mathsf{E}\{S_{t_1(\alpha)-1}|\frak{A}_1\}-\mathsf{E}\{S_{\tau_1(\alpha)}|\frak{A}_1\}$.

Let $\chi$ denote the length of remaining intervals partitioned on
$L_\alpha+1$ parts, so this remaining time interval is represented
as
$$[0,\chi)=\bigcup\limits_{j=1}^{L_\alpha+1}I_j,$$
where
$$
I_1=[1, t_1-1), I_2=[\tau_1, t_2-1), I_3=[\tau_{2},t_3-1),\ldots,
$$
$$
I_{L_\alpha}=[\tau_{L_\alpha-1},t_{L_\alpha}-1),
I_{L_\alpha+1}=[\tau_{L_\alpha},\tau).
$$


Let $\eta_1^-$, $\eta_2^-$,\ldots, $\eta_{L_\alpha+1}^-$ denote the
numbers of random variables $X_i$ in the corresponding time
intervals $I_1$, $I_2$,\ldots, $I_{L_\alpha+1}$ that take
non-positive value, and, respectively, let $\eta_1^+$,
$\eta_2^+$,\ldots, $\eta_{L_\alpha+1}^+$ denote the numbers of
random variables $X_i$ in the corresponding time intervals $I_1$,
$I_2$,\ldots, $I_{L_\alpha+1}$ that take positive value. Next, set
$\eta^-:=\eta_1^-+\ldots+\eta_{L_\alpha+1}^-$, and, respectively,
$\eta^+:=\eta_1^++\ldots+\eta_{L_\alpha+1}^+$.\\

Assume that the original random variables $X_t^+$ and $X_t^-$ all
are renumbered after the above time scale procedure and follow in
the ordinary order. Then $S_{\eta^-}^-$ and $S_{\eta^+}^+$ can be
written as
\begin{equation}\label{z1}
S_{\eta^-}^-=S_{\eta_1^-+\eta_2^-+\ldots+\eta_{L_\alpha+1}^-}^-,
\end{equation}
and
\begin{equation}\label{z3}
S_{\eta^+}^+=S_{1+\eta_1^++\eta_2^++\ldots+\eta_{L_\alpha+1}^+}^+.
\end{equation}
The random variable $\eta^+$ includes a positive random variable
$X_1$ that starts a random walk in the interval $I_1$. (Recall that
our convention was $X_1>0$ and we are going to prove that
$\mathsf{E}\{L_\alpha|X_1>0\}=\frac{a-\mathsf{E}\{S_\tau|X_1>0\}}{a-b}$.)
For this reason there is the difference in the notation for
$S_{\eta^+}^+$ in \eqref{z3} compared to that for $S_{\eta^-}^-$ in
\eqref{z1}. There is extra 1 in the subscript line of the right-hand
side of \eqref{z3}.

As well, there is the difference between the initial value of the
random walk and the moment of stopping $S_\tau$. This difference is
equal to $-S_\tau$, and its expected value is
$-\mathsf{E}\{S_\tau|X_1>0\}$. Hence,
\begin{equation}\label{z7}
\mathsf{E}\{S_{\eta^+}^+-S_{\eta^-}^-|X_1>0\}=a+b\mathsf{E}\{L_\alpha|X_1>0\}-\mathsf{E}\{S_\tau|X_1>0\},
\end{equation}
where the second term on the right-hand side of \eqref{z7},
$b\mathsf{E}\{L_\alpha|X_1>0\}$, is calculated due to Wald's
identity (e.g. Feller \cite{Feller2}, p.384).

Applying Wald's identity once again, we obtain
$$
\mathsf{E}\{L_\alpha|X_1>0\}\mathsf{P}\{X_t>0\}\mathsf{E}X_t^+={a+b\mathsf{E}\{L_\alpha|X_1>0\}-\mathsf{E}\{S_\tau|X_1>0\}},
$$
and hence, due to the fact that
$\mathsf{E}X_1^+=\frac{a}{\mathsf{P}\{X_t>0\}}$ we arrive at
$$\mathsf{E}\{L_\alpha|X_1>0\}=\frac{a-\mathsf{E}\{S_\tau|X_1>0\}}{a-b}.$$
The first part of the theorem is proved.

Let us now prove that there exists an infinitely increasing sequence
of values $\alpha_1$, $\alpha_2$, \ldots such that Assumption 1 is
satisfied for these values as well, and moreover, for all
$n=1,2,\ldots$
$$
\begin{aligned}
&\mathsf{E}\{S_{t_1(\alpha_n)-1}|\frak{A}_1(\alpha_n)\}-\mathsf{E}\{S_{\tau_1(\alpha_n)}|\frak{A}_1(\alpha_n)\}\\
&=\mathsf{E}\{S_{t_1(\alpha)-1}|\frak{A}_1(\alpha)\}-\mathsf{E}\{S_{\tau_1(\alpha)}|\frak{A}_1(\alpha)\}\\
&:=b.
\end{aligned}
$$
Assuming that the event $\frak{A}_1$ occurs, denote by
$\mathcal{S}_\alpha$ the set of values $S_{t_1(\alpha)-1}$, denote
$\tau(x)=\inf\{t>t_1(\alpha): S_{t}\leq
S_{t_1(\alpha)-1}|S_{t_1(\alpha)-1}=x\}$, and denote by
$\mathcal{Z}_\alpha$ the set of all values of $S_{\tau(x)}$. Then
$x$ is an initial (non-random) point of a new random walk and
$\tau(x)$ is a random stopping time, and during the time interval
$[t_1(\alpha)-1, \tau(x)]$ the behavior of this random walk is the
same as that of the original random walk that starts at zero. Hence,
\begin{equation}\label{OS10}
\mathsf{E}\{S_{t_1(\alpha+x)-1}|S_{t_1(\alpha)-1}=x,
\frak{A}_1(\alpha+x)\}=\mathsf{E}\{S_{t_1(\alpha)-1}+x|\frak{A}_1(\alpha+x)\},
\end{equation}
for all possible values $x\in\mathcal{S}_\alpha$.

By the total expectation formula, we obtain:
$$
\mathsf{E}\{S_{t_1(\alpha+\mathsf{E}S_{t_1(\alpha)-1})-1}|\frak{A}_1(\alpha+\mathsf{E}S_{t_1(\alpha)-1})\}
=2\mathsf{E}\{S_{t_1(\alpha)-1}|\frak{A}_1(\alpha)\}.
$$
Hence,
$\mathsf{E}\{S_{t_1(\alpha_1)-1}|\frak{A}_1(\alpha_1)\}=2\mathsf{E}\{S_{t_1(\alpha)-1}|\frak{A}_1(\alpha)\}$
for $\alpha_1=\alpha+\mathsf{E}S_{t_1(\alpha)-1}$. For the level
$\alpha_1$, the same properties as those for the original level
$\alpha$ are satisfied. That is,
\begin{equation}\label{z9}
\mathsf{E}\{S_{t_2(\alpha_1)-1}|\frak{A}_2(\alpha_1)\}=\mathsf{E}\{S_{t_1(\alpha_1)-1}|\frak{A}_1(\alpha_1)\}.
\end{equation}

Respectively, with the coupling arguments we obtain
$$
\mathsf{E}\{S_{\tau_2(\alpha_1)}|\frak{A}_2(\alpha_1)\}
=\mathsf{E}\{S_{\tau_1(\alpha_1)}|\frak{A}_1(\alpha_1)\}.
$$

Similar arguments of the induction enable us to obtain the relation
$$
\alpha_{i+1}=\alpha_i+\mathsf{E}\{S_{t_1(\alpha)-1}|\frak{A}_1(\alpha)\},
$$
where all required properties of the stopping times are satisfied.
This completes the proof of the second part of the theorem as well,
and totally completes the proof of the theorem.
\end{proof}

\begin{rem} Assumptions $\mathsf{E}\{S_{t_{1}(\alpha)-1}|\frak{A}_1(\alpha)\}>0$ and
  $\mathsf{E}\{S_{\tau_1(\alpha)}|\frak{A}_1(\alpha)\}>0$
are important. If at least one of them is not satisfied, then the
expected number of level-crossings
need not be equal to the value that obtained in the statement of the
theorem, because in this case, equality \eqref{z7} is not valid,
since the level 0 is ``overlapped" by negative value(s) in \eqref{1}
or/and \eqref{2}, and one cannot use the value
$\mathsf{E}\{S_\tau|X_1>0\}$ in expression \eqref{z7}.

%
%
\end{rem}

The result for $P$-symmetric random walk follows from Theorem
\ref{thm2} as a corollary.

\begin{cor}\label{cor1}
Let $S_t$ be a $P$-symmetric random walk, and let Assumption
\ref{asum1} be satisfied for some $\alpha>0$. Then, there exists an
infinitely increasing sequence of levels $\alpha_n$ such that
\begin{equation*}\label{(13)}
\mathsf{E}L_{\alpha_n}=\mathsf{E}L_{\alpha}=\frac{1}{2}\cdot\frac{a-\mathsf{E}\{S_\tau|X_1>0\}}{a-b}\mathsf{P}\{X_1\neq0\}.
\end{equation*}
where $a=\mathsf{E}\{X_1|X_1>0\}$, and
$b=\mathsf{E}\{S_{t_1(\alpha)-1}|\frak{A}_1(\alpha)\}-\mathsf{E}\{S_{\tau_1(\alpha)}|\frak{A}_1(\alpha)\}$.

\end{cor}

\begin{proof} Indeed, the relation
$\mathsf{P}\{X_1>0\}=\mathsf{P}\{X_1<0\}$  enables us to conclude
that $\mathsf{P}\{X_1>0\}=\frac{1}{2}\mathsf{P}\{X_1\neq0\}$. Hence,
the result follows as a reformulation of Theorem \ref{thm2}.
\end{proof}

In the particular case where $\mathsf{P}\{X_1=0\}=0$, for
$P$-symmetric random walks satisfying Assumption \ref{asum1} we have
$\mathsf{E}L_{\alpha_n}=\frac{1}{2}\cdot\frac{a-\mathsf{E}\{S_\tau|X_1>0\}}{a-b}$
for all values $\alpha_n$ defined in the proof of Theorem
\ref{thm2}.

The following theorem demonstrates application of Theorem \ref{thm2}
to a special type of $E$-symmetric random walks.

\begin{thm}\label{thm4} Suppose that $X_t^+$ takes only a single positive value $d$,
while $X_t^-$ takes the values 0, $-d$, $-2d$, \ldots (some of them
can have probability 0). Then, for all
$\alpha>\max[d,-\mathsf{E}\{S_\tau|X_1>0\}]$ we have
$$
\mathsf{E}L_\alpha=\frac{d-\mathsf{E}\{S_\tau|X_1>0\}}
{d-\mathsf{E}\{S_\tau|X_1>0\})\mathsf{P}\{X_1>0\}}\mathsf{P}\{X_1>0\}.
$$
\end{thm}

\begin{proof} Indeed, it is readily seen that for
$\alpha>\max[d,-\mathsf{E}\{S_\tau|X_1>0\}]$ Assumption \ref{asum1}
is satisfied, since if the event $\frak{A}_i(\alpha)$ occurs
($i=1,2,\ldots$, then $S_{t_i(\alpha)}=d\inf\{m: md\geq\alpha\}$,
$S_{t_i(\alpha)-1}$ is positive, and
$$
\mathsf{E}\{S_{\tau_2(\alpha)}|S_{\tau_1(\alpha)},\frak{A}_2(\alpha)\}=\mathsf{E}\{S_{\tau_2(\alpha)}|
\frak{A}_2(\alpha)\}
=\mathsf{E}\{S_{\tau_1(\alpha)}|\frak{A}_1(\alpha)\}>0.
$$
In addition,
$$
\begin{aligned}
\mathsf{E}\{S_{\tau_1(\alpha)}|\frak{A}_1(\alpha)\}&=\mathsf{E}\{S_{t_1(\alpha)}|\frak{A}_1(\alpha)\}
-\mathsf{E}X_t^-\mathsf{P}\{X_t\leq0\}-
(-\mathsf{E}\{S_\tau|X_1>0\})\mathsf{P}\{X_t>0\}\\
&=\mathsf{E}\{S_{t_1(\alpha)}|\frak{A}_1(\alpha)\}-d-
(-\mathsf{E}\{S_\tau|X_1>0\})\mathsf{P}\{X_t>0\}.
\end{aligned}
$$
Hence, the constant $b$ is
$$
\begin{aligned}
b&=\mathsf{E}\{S_{t_1(\alpha)-1}|\frak{A}_1(\alpha)\}-(\mathsf{E}\{S_{t_1(\alpha)}|\frak{A}_1(\alpha)\}-d-
(-\mathsf{E}\{S_\tau|X_1>0\})\mathsf{P}\{X_t>0\})\\
&=-d+d+\mathsf{E}\{S_\tau|X_1>0\}\mathsf{P}\{X_t>0\}\\
&=\mathsf{E}\{S_\tau|X_1>0\})\mathsf{P}\{X_t>0\},
\end{aligned}
$$
and we arrive at the statement of the theorem.
\end{proof}

\begin{cor} In the case where $X_t^+$ takes only one single positive value
$d$, and $X_t^-$ takes the only values $\{0, -d\}$ from Theorem
\ref{thm4} we obtain
$$
\mathsf{E}L_\alpha=\frac{1}{2}\mathsf{P}\{X_1\neq0\},
$$
which is correct for all $\alpha>0$.
\end{cor}

\begin{proof} Indeed, in this case $\mathsf{E}\{S_\tau|X_1>0\}=0$.
Hence, the result follows from Theorem \ref{thm4}.
\end{proof}

Let us now discuss purely symmetric random walks. In the case of
these random walks, Corollary \ref{cor1} is simplified as follows.
If Assumption \ref{asum1} is satisfied, then
$\mathsf{E}\{S_{t_1-1}|\frak{A}_1(\alpha)\}$ and
$\mathsf{E}\{S_{\tau_1}|\frak{A}_1(\alpha)\}$ must be equal and we
have the following statement.

\begin{thm}\label{cor2}
Let $S_t$ be a purely symmetric random walk, and let Assumption
\ref{asum1} be satisfied for some $\alpha_0>0$. Then, for all levels
$\alpha\geq\alpha_0$,
\begin{equation*}\label{(14)}
\mathsf{E}L_{\alpha}=\frac{1}{2}\cdot\frac{a-\mathsf{E}\{S_\tau|X_1>0\}}{a}\mathsf{P}\{X_1\neq0\}.
\end{equation*}
\end{thm}

\begin{proof}
For any $\alpha>\alpha_0$, because the random walk is purely
symmetric,
$$
\mathsf{P}\{S_{t_i(\alpha)-1}\in
B|\frak{A}_i\}=\mathsf{P}\{S_{\tau_i(\alpha)}\in B|\frak{A}_i\}
$$
for all $i=1,2,\ldots$ and any Borel set $B\in \mathbb{R}^1$. Hence,
$$
b_\alpha=\mathsf{E}\{S_{t_i(\alpha)-1}|\frak{A}_i\}-\mathsf{E}\{S_{\tau_i(\alpha)}|\frak{A}_i\}=0.
$$
Applying the arguments of the proof of Theorem \ref{thm2} we arrive
at the desired conclusion.
\end{proof}

\begin{rem}\label{rem1}
From Theorem \ref{cor2} we conclude as follows. Let $\{S_t, r\}$ be
a family of all purely symmetric random walks with the given
probability $r=\mathsf{P}\{X_1\neq0\}$. Apparently, the minimum of
$\mathsf{E}L_\alpha$ for the family $\{S_t, r\}$ is achieved in the
case when $\mathsf{E}\{S_\tau|X_1>0\}=0$, and for all $\alpha>0$ we
obtain:
$$
\min_{\{S_t, r\}} \mathsf{E}L_\alpha=\frac{1}{2}r.
$$
\end{rem}

\begin{exam}\label{exam4} Let us return to Example \ref{exam1}. For the purely symmetric
random walk that specified there, it was shown
$\mathsf{E}L_1=\frac{1}{2}$. Let us evaluate $\mathsf{E}L_\alpha$,
for all $\alpha\geq2$. Notice that
$\mathsf{P}\{S_1=1|X_1>0\}=\mathsf{P}\{S_1=2|X_1>0\}=\frac{1}{2}$,
and
$$
\mathsf{P}\{S_t=1|1\leq S_i\leq2,
i=1,2,\ldots,t\}=\mathsf{P}\{S_t=2|1\leq S_{i}\leq 2,
i=1,2,\ldots,t\}=\frac{1}{2}.
$$
Hence, for $\alpha=2$, $\mathsf{E}\{S_{t_1(2)-1}|\frak{A}_1\}=
\mathsf{E}\{S_{\tau_1(2)}|\frak{A}_1\}=\frac{2}{3}>0$,  and
$\mathsf{E}\{S_\tau|X_1>0\}=-\frac{1}{3}$. Keeping in mind that
$a=\frac{3}{2}$, by Theorem \ref{cor2}, for all $\alpha\geq2$ we
have
$$
\mathsf{E}L_\alpha=\frac{1}{2}\cdot\frac{\frac{3}{2}+\frac{1}{3}}{\frac{3}{2}}=\frac{11}{18}.
$$
\end{exam}

\section{Application to queuing theory}\label{Sect5}
\noindent In this section, the application of level-crossings in
random walks is demonstrated for elementary queuing problems in a
not traditional formulation.

Consider a $M^X/M^Y/1/N$ queuing system in which the expected
interarrival time of some random quantity $X$ (which we associate
with an arrival of a \textit{customer} for convenience) is equal to
$\frac{1}{\lambda}$, the expected service time of some random
quantity $Y$ is equal to $\frac{1}{\mu}$. $X$ characterizes a
`weight' of a customer (say mass), and $Y$ characterizes a
`capacity' (say mass) of a service. The random variables $X$ and $Y$
are assumed to be positive random variables and not necessarily
integer, and $N$ is assumed to be a positive real number in general.

Let $X_1$, $X_2$,\ldots denote consecutive weights of customers
having the same distribution as $X$, and let $Y_1$, $Y_2$,\ldots
denote the service capacities having the same distribution as $Y$.
The sequences $\{X_1, X_2, \ldots\}$ and $\{Y_1, Y_2, \ldots\}$ are
assumed to be independent and to consist of independently and
identically distributed random variables.

The \textit{full rejection policy} is supposed. In this policy, if
at the moment of arrival of a customer the capacity $N$ of the
system is exceeded, then a customer (with his/her entire weight)
leaves the system without any service.

Let $L_N$ denote the number of losses during a busy period. In the
theorem below we study the expected number of losses during a busy
period. For this specific queuing system, the behaviour of the
number of losses during a busy period differs from that of the
number of level-crossings in the associated random walk. However,
the arguments that were used in the proof of Theorem \ref{thm2} are
applied here as well.

\begin{thm}\label{thm3}
Assume that $ \lambda\mathsf{E}X=\mu\mathsf{E}Y $. Then for any
$N>\mathsf{E}X$
 and any nontrivial random variable $Y$ (i.e. taking at least
two positive values) we have the inequality $\mathsf{E}L_N>1$.
\end{thm}

\begin{proof} Let $p=\frac{\lambda}{\lambda+\mu}$, and let
$q=\frac{\mu}{\lambda+\mu}$. Denote by $S_t$ an associated random
walk in which $S_0=0$ and $S_1=X_1$ and the following values are
$S_t=S_{t-1}+W_t$, $t=2,3,\ldots$, where
$W_t=X_t\mathsf{I}(A)-Y_t\mathsf{I}(\overline{A})$, $A$ and
$\overline{A}$ are opposite events. The event $A$ occurs with the
probability $p$ and the event $\overline{A}$ occurs with the
complementary probability $q$. Thus, the associated random walk is a
$E$-symmetric random walk. Let $t_1$, $t_2$,\ldots, $t_{L_N}$ be the
moments when the customers are lost from the system.

Then the difference between the structure of a $E$-symmetric random
walk and the queuing process is as follows. The random variables
$S_{t_i-1}$ and $S_{\tau_i}$ have generally different expectations
in relations \eqref{1} and \eqref{2}, and the difference between
them is denoted $b$. The expected number of level-crossings in
Theorem \ref{thm2} is expressed via this quantity $b$. In the case
of the queuing system considered here, we have the equality
$S_{t_i-1}=S_{\tau_i}$, where $S_{t_i-1}$ is the value of the
queuing capacity before the moment when the loss from the system
occurs, and $S_{\tau_i}$ is the value of capacity after the loss.
More specifically, in the case of queuing system the time moments
$t_i-1$ and $\tau_i$ are the same, and being compared with those of
associated random walk they can be considered as coupled. In other
words, the random walk is ``cut" in the points $t_i-1$ and they are
``coupled" with the points $\tau_i$. Then apparently,
$\mathsf{E}\{S_{t_i-1}|\text{the} \ i\text{th loss
occurs}\}=\mathsf{E}\{S_{\tau_i}|\text{the} \ i\text{th loss
occurs}\}$, and the application of the same arguments as those in
the proof of Theorem \ref{thm2} in the given case should be made
with $b=0$. It is only taken into account that for any nontrivial
random variable $Y$ (taking at least two positive values) we have
$\mathsf{E}\{S_\tau|X_1>0\}<0$, where $S_\tau$ is the stopping time
of the ``cut" random walk as explained above. The physical meaning
of the inequality $\mathsf{E}\{S_\tau(N)|X_1>0\}<0$ is associated
with the case that the last service batch in a busy period is
incomplete.

As in the case of associated $E$-symmetric random walk, for the
value $N$ the inequality $N>\mathsf{E}X$ should be taken into
account in order to guarantee the condition
$\mathsf{E}\{S_{t_1-1}|\text{the first loss occurs}\}>0$. Then,
similarly to the main result of Theorem \ref{thm2} we have:
\begin{equation}\label{z8}
\mathsf{E}L_N=\frac{a-\mathsf{E}\{S_\tau|X_1>0\}}{a},
\end{equation}
where the only difference is that $S_\tau$ is related to the ``cut"
random walk, and unlike in the usual random walk now in depends on
$N$ as well, i.e. $S_\tau=S_\tau(N)$. Since
$\mathsf{E}\{S_\tau(N)|X_1>0\}$ is strictly negative for any $N$,
from \eqref{z8} we finally obtain $\mathsf{E}L_N>1$.
\end{proof}

\begin{rem}
According to Theorem \ref{thm3}, for any non-trivial random variable
$Y$ we have $\mathsf{E}L_N>1$. (The values $\mathsf{E}L_N$ generally
depend of $N$, because for different $N$ the values
$\mathsf{E}\{S_\tau|X_1>0\}$ can be different.) In the case of the
$M^X/M/1/N$ queuing system where $X$ is a positive integer random
variable and $Y=1$ we have $\mathsf{E}L_N=1$ for all $N\geq0$,
because in this case $\mathsf{E}\{S_\tau|X_1>0\}=0$. This result,
remains correct for $M^X/GI/1/N$ queuing systems with generally
distributed service times (see \cite{Abramov}, \cite{Righter} and
\cite{Wolff}). However, in the case of the $M^X/M/1/N$ queuing
system where $X$ is a positive \textit{continuous} random variable
and $Y=1$ the equality $\mathsf{E}L_N=1$ does not hold.  In this
case we have the inequality $\mathsf{E}L_N>1$, because when $X$ is a
continuous random variable, the last service batch in a busy period
is incomplete, and we have $\mathsf{E}\{S_\tau|X_1>0\}<0$.
\end{rem}

\begin{exam}Consider a very simple example of the problem where $X$ takes
discrete values $0.1$ and $0.2$ with the equal probability
$\frac{1}{2}$ and $Y$ takes the same values $0.1$ and $0.2$ with the
same probability $\frac{1}{2}$ each. The equal values of $\lambda$
and $\mu$ are not a matter, let they both be equal to 1. Let $N=1$.
For this specific example, the value $\mathsf{E}L_N$ can be
evaluated similarly to that of Example \ref{exam4}. According to
Corollary \ref{cor2}, the value $\mathsf{E}L_N$ coincides with the
expected number of level-crossings in associated random walk, given
that its first jump is positive. So,
$$
\mathsf{E}L_N=\frac{\frac{15}{100}+\frac{1}{30}}{\frac{15}{100}}=\frac{11}{9}.
$$
Since the arrival and departure processes are symmetric,
$\mathsf{E}L_N$ is the same for all $N$ as in the associated random
walk.
\end{exam}

\begin{exam} Consider the example of the above $M^X/M^Y/1/N$ queuing
system where parameters $\lambda$ and $\mu$ both are equal to 1, the
random variable $X$ is generally distributed with mean $0.15$ and
the random variable $Y$ is exponentially distributed with the same
mean $0.15$ and $N=1$. In this case $\mathsf{E}L_N$ is independent
of $N$ as well and obtained exactly. Indeed, according to the
property of the lack of memory of the exponential distribution in
the associated random walk we have
$\mathsf{E}\{S_\tau|X_1>0\}=-0.15$, and hence,
$$
\mathsf{E}L_N=\frac{.15+.15}{.15}=2.
$$
\end{exam}

\section{Concluding remarks}\label{Sect6}
\noindent In the present paper we studied level-crossings of
symmetric random walk. We addressed the questions formulated in
Section \ref{Sect3}. We showed that under specified conditions given
by Assumption \ref{asum1} for $E$-symmetric random walks there
exists the increasing sequence of levels such that the expected
number of level-crossings remains the same. We obtained the expected
number of level-crossings for special class of $E$-symmetric random
walks (Theorem \ref{thm4}). For purely symmetric random walks we
established a more general result saying that the expected number of
level-crossings remains the same for all levels that greater some
initial value $\alpha_0$. It follows from Theorem \ref{thm2} and
Corollary \ref{cor1} that for $E$- and $P$- symmetric random walks
the value $\mathsf{P}\{X_1>0\}$ is not the minimum of the expected
number of level-crossings within these classes. We showed, however
(see Remark \ref{rem1}), that within the class of purely symmetric
random walks, the expected number of level-crossings is not smaller
than $\frac{1}{2}\mathsf{P}\{X_1\neq0\}$, and this lower bound is
within the class of these random walks. Thus we addressed the second
question formulated in Section \ref{Sect3}. Finally, we obtained
non-trivial results for the expected number of losses during a busy
period of loss queueing systems.

There is a number of possible directions for the future work. One of
them can be associated with the case when Assumption \ref{asum1} is
not satisfied. A new study, stimulated by Example \ref{exam7}, can
be provided under the following assumption.

\begin{asum}
Assume that
$\mathsf{E}\{S_{t_{1}(\alpha)-1}|\frak{A}_1(\alpha)\}>0$,
  $\mathsf{E}\{S_{\tau_1(\alpha)}|\frak{A}_1(\alpha)\}>0$,
and
\begin{equation*}\label{OS2.1}
\mathsf{E}\{S_{\tau_2(\alpha)}|\frak{A}_2(\alpha)\}=\mathsf{E}\{S_{\tau_1(\alpha)}|\frak{A}_1(\alpha)\}.
\end{equation*}
\end{asum}

\section*{Acknowledgement}\noindent
The author thanks Professor Daryl Daley for relevant questions
during my talk in a seminar in Swinburne University of Technology,
which helped me to substantially revise this paper.

\end{document}